\RequirePackage{fix-cm}
\documentclass{svjour3}                     
\smartqed  
\usepackage{newcmd}
\usepackage[fleqn]{amsmath}
\usepackage{graphicx}
\usepackage{amssymb}
\usepackage{float}
\usepackage{mathrsfs}
\usepackage{comment}
\usepackage{mathtools}
\usepackage{yfonts}
\usepackage{subcaption}
\newtheorem{assumption}[theorem]{Assumption}
\usepackage{amsfonts}
\usepackage{newcmd}
\DeclareMathAlphabet{\mathpzc}{OT1}{pzc}{m}{it}
\DeclareSymbolFont{matha}{OML}{txmi}{m}{it}
\DeclareMathSymbol{\varv}{\mathord}{matha}{118}
\usepackage[dvipsnames]{xcolor}

\usepackage{setspace}
\begin{document}

\title{Discrete-time Rigid Body Pose Estimation based on Lagrange-d'Alembert principle 
}

\titlerunning{Discrete-time Rigid Body Pose Estimation}        

\author{Maulik Bhatt$^{*}$ \and Srikant Sukumar \and Amit K Sanyal}


\institute{M. Bhatt (Corresponding Author) \at
              Aerospace Engineering, Indian University of Illinois at Urbana-Champaign, IL, USA. \\
              Tel.: +1 (217) 377-8235\\
              \email{mcbhatt2@illinois.edu, maulikbhatt585@gmail.com}           
           \and
           S. Sukumar \at
              Systems and Control Engineering, Indian Institute of Technology Bombay, 400076, India.
             \and
             A.K. Sanyal \at
             Mechanical and Aerospace Engineering, Syracuse University, Syracuse, NY, USA.
}

\date{Received: date / Accepted: date}

\maketitle

\begin{abstract}
The problem of rigid body pose estimation is treated in discrete-time via discrete Lagrange-d'Alembert principle and discrete Lyapunov methods. The position and attitude of the rigid body are to be estimated simultaneously with the help of vision and inertial sensors. For the discrete-time estimation of pose, the continuous-time rigid body kinematics equations are discretized appropriately. We approach the pose estimation problem as minimising the energies stored in the errors of estimated quantities. With the help of measurements obtained through optical sensors, artificial rotational and translation potential energy-like terms have been designed. Similarly, artificial rotational and translation kinetic energy-like terms have been devised using inertial sensor measurements. This allows us to construct a discrete-time Lagrangian as the difference of the kinetic and potential energy like terms, to which a Lagrange-d'Alembert principle is applied to obtain an optimal pose estimation filter. The dissipation terms in the optimal filter are designed through discrete-Lyapunov analysis on a suitably constructed Morse-Lyapunov function and the overall scheme is proven to be almost globally asymptotically stable. The filtering scheme is simulated using noisy sensor data to verify the theoretical properties.
\end{abstract}

\section{Introduction}

The pose of a rigid body with respect to a frame is a transformation from a body-fixed frame to an inertial frame. The pose encapsulates the position of the center of mass and orientation of the rigid body. Estimation of the pose of a rigid body has various applications in the control of spacecraft, ground vehicles, underwater vehicles for example. Generally, the position and attitude are estimated with the help of onboard inertial sensors coupled with a dynamic model. However, when available, external measurements such as GPS or tracking data of multiple points on the body are also used for pose estimation (\cite{amelin2014algorithm,vasconcelos2008nonlinear,vertechy2007accurate}). Some techniques combine inertial sensors, vision sensors and external measurements to estimate the rigid body pose. Furthermore, it is common in several applications to operate in GPS denied environments. Therefore, an estimation scheme relying on inertial and vision sensors with proven stability properties and a large domain of attraction is necessary. Additionally, robustness to uncertainties and noise is required.\\

In recent times, several stable nonlinear estimators evolving on non-Euclidean spaces such as $\SO$ or $\SE$ have been presented with provably large domain of attraction. A landmark-based nonlinear pose observer is proposed in \cite{vasconcelos2007landmark} which is almost globally exponentially stable on $\SE$. The pose estimation scheme in \cite{rehbinder2003pose} uses line-based dynamic vision and inertial sensors to provide a locally convergent attitude observer and subsequently a position estimator. A quaternion based pose estimator is presented in \cite{filipe2015extended} where cost functions based on estimation errors are constructed in discrete-time and minimized to obtain a filtering scheme. The attitude estimation problem based on vector measurements was first proposed as an optimization problem on $\SO$ by Wahba in \cite{wahba1965least}. The cost function is known as Wahba's cost function. In \cite{vasconcelos2010nonlinear}, the authors devise a pose estimator using a Lyapunov function defined as the difference between the estimated and the measured landmark coordinates. For attitude estimation, similar ideas are used in \cite{mahony2008nonlinear,zamani2013minimum,izadi2014rigid,bhatt2020rigid}. In \cite{izadi2016rigid}, the authors applied the Lagrange-d'Alembert principle to a Lagrangian constructed through state estimation errors to obtain an optimal filtering scheme for the rigid body pose. However, the work in \cite{izadi2016rigid} provides a continuous-time pose estimator by applying the (continuous time) Lagrange-d'Alembert principle to a Lagrangian. The estimator is then discretized for numerical implementation, which voids the theoretical guarantee of asymptotic stability provided by the continuous-time estimator. In this work we obtain a discrete-time pose estimation scheme by applying the discrete Lagrange-d'Alembert principle on a discrete time Lagrangian. Furthermore, we also prove guaranteed asymptotic stability by performing the discrete-Lyapunov analysis of the system and prove almost global asymptotic stability. 
Discrete-time observers for \emph{only} attitude with stability properties can be found in \cite{bhatt2020optimal,bhatt2020rigid}. \\

In this paper, we derive an optimal pose estimation scheme by minimizing the ``energy'' stored in the state estimation errors. A discrete-time Lagrangian has been devised and the discrete Lagrange-d'Alembert principle from variational mechanics \cite{marsden2001discrete} is employed to obtain an \emph{optimal} filtering scheme. It is then proven to be \emph{almost globally asymptotically stable} via discrete-Lyapunov analysis. The pose of the rigid body is expressed in $\SE$ without employing any local coordinates (such as Euler angles or quaternion) and hence globally non-singular. Furthermore, the estimation scheme presented here relies only on on-board sensor data. We also do not make assumptions on the statistical properties of the measurement noise as is usually the case for Kalman filter-based estimation schemes. \\

This paper is organized as follows. In the section \ref{sec:2}, relevant notations are introduced and the procedure to estimate rigid body pose using measurements is explained. The continuous-time rigid body kinematics has been discretized in section \ref{sec:3}. Section \ref{sec:4} contains the application of variational mechanics to obtain a filter equation for pose estimation. The filter equations obtained in section \ref{sec:4} are proven to be asymptotically stable using the discrete-time Lyapunov method in section \ref{sec:5}. Filter equations are numerically verified with realistic measurements (corrupted by bounded noise) in section \ref{sec:6}. Finally, section \ref{sec:7} presents the concluding remarks and possible directions future work.


\section{Notation and Problem Formulation}\label{sec:2}

\subsection{Notation and Preliminaries}

We define the trace inner product on $\bR^{m\times n}$ as
\begin{equation*}
    \langle A_1,A_2\rangle := \text{trace}(A_1\T A_2).
\end{equation*}
The group of orthogonal frame transformations on $\bR^3$ is defined by $\mathrm{O}(3) := \{Q \in \bR^{3\times 3} \; | \; \text{det}(Q)=\pm 1\}$. The Special orthogonal group on $\bR^3$ is denoted as $\SO$ and defined as $\SO := \{R \in \bR^{3\times 3} \; | \; R\T R = RR\T = I_{3}\}$. Let there be some $R \in \SO$ and $b\in\bR^3$. The corresponding Lie algebra is denoted as $\mathfrak{so}(3) := \{ M \in \bR^{3\times 3} \; | \; M + M\T = 0\}$. The Special Euclidean group, $\SE$, corresponds to the set of all $4 \times 4$ transformation matrices of the form,
\begin{equation*}
    \SE := \left\{ \begin{pmatrix}R & b \\ 0 & 1  \end{pmatrix} \in \bR^{4\times4} \;\bigg\vert\; R \in \SO \mbox{ and } b \in \bR^3 \right\}.
\end{equation*} 
Let $(\cdot)^\times:\bR^3 \rightarrow \mathfrak{so}(3) \subset \bR^{3\times3}$ be the skew-symmetric matrix cross-product operator denoting the vector space isomorphism between $\bR^3$ and $\mathfrak{so}(3)$:
\begin{equation*}
    v^\times = {\begin{bmatrix} v_1 \\ v_2 \\ v_3 \end{bmatrix}}^{\times} := \begin{bmatrix} 0 & -v_3 & v_2 \\ v_3 & 0 & -v_1\\ -v_2 & v_1 & 0 \end{bmatrix}.
\end{equation*}
Further, let $\text{vex}(\cdot):\mathfrak{so}(3)\rightarrow\bR^3$ be the inverse of $(\cdot)^\times$. The map $\exp{(\cdot)}:\mathfrak{so}(3)\rightarrow \SO$ is defined as
\begin{equation*}
    \exp{(M)} := \sum_{i=0}^{\infty}\frac{1}{k!}M^k.
\end{equation*}
We define $Ad:\SO\times\mathfrak{so}(3)\rightarrow\mathfrak{so}(3)$ as
\begin{equation*}
    Ad_R\Omega^\times := R\Omega^\times R\T = (R\Omega)^\times.
\end{equation*}
For the remainder of the article, the phrase ``consider the time interval $[t_0,T]$'', indicates that the estimation process will be carried out over the time interval $[t_0,T]$ and is divided into $N$ equal sub-intervals $[t_i,t_{i+1}]$ for $i = 0,1,\dots,N$ with $t_N = T $. The time step size is denoted as, $h := t_{i+1} - t_i$. Further, given a state $z(t)$, $z_i:=z(t_i)$. Give a quantity $\beta$, $\beta^m$ denotes its measurement through an on-board sensor.

\subsection{Navigation using optical and inertial sensors}
Assume that a rigid body exhibits rotational and translation motion in an environment. The pose estimation of the rigid body implies estimation of the orientation and position of a frame $S$, fixed to the rigid body center of mass with respect to some inertial frame $O$, fixed to the observed environment as shown in figure \ref{fig:1}. Let $R \in \SO$ be the rotation matrix from $S$ to $O$ and $b \in \bR^3$ be the location of the origin of $S$ in the frame $O$. We can write the pose $g \in \SE$ of the rigid body as,
\begin{equation}
    g := \begin{bmatrix} R & b \\ 0 & 1 \end{bmatrix}.
\end{equation}
If there exists a column vector $\psi = [x \; y \; z]\T \in \bR^3$, then it can be represent as a column vector $\gamma=[x \; y \; z \; 1]\T$ in $\bR^3$ as a subspace of $\bR^4$. Furthermore, $g\in\SE$ acts on this vector by a combination of rotation and translation as follows:
$g \gamma= R\psi + b$. \\
\begin{figure}
    \centering
    \includegraphics[scale=1.1]{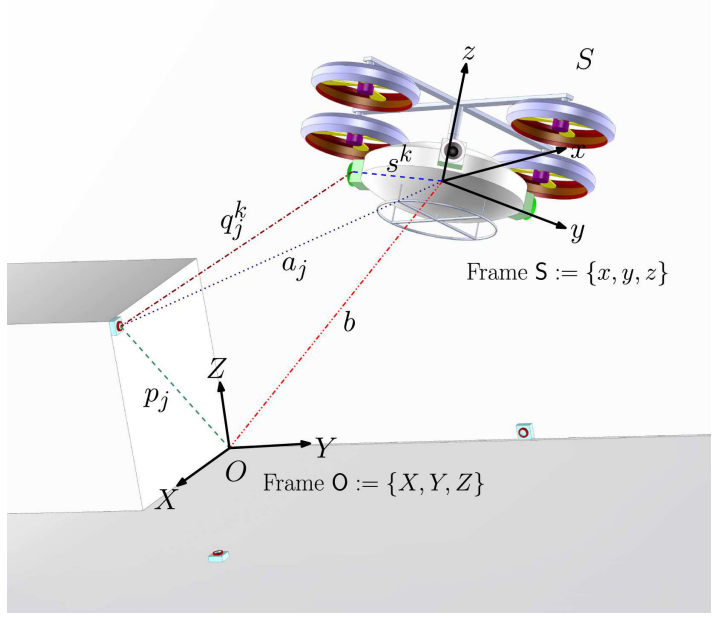}
    \caption{Inertial landmarks in frame O as observed from vehicle S with optical measurements. $O$ - inertial frame, $S$-body-fixed frame, $b$ - position of the center of mass of the body, $p_j$ - position of the $j^{\text{th}}$ beacon in frame $O$, $a_j$ - position of the $j^{\text{th}}$ beacon in frame $S$, $s^k$ - position of the $k^{\text{th}}$ optical sensor in frame $S$, $q_j^k$ - range from $k^{\text{th}}$ optical sensor to $j^{\text{th}}$ beacon}
    \label{fig:1}
\end{figure}

Assume that there are $r^g$ number of inertial vectors (such as gravity vector, magnetic field vector) whose locations in the frame $O$ are known (denoted as $e_j$ for $j= 1,2, \ldots, r^g$) and can be measured in the frame $S$ via inertial sensors (denoted as $e_j^s$). Furthermore, let there be beacons placed with their position vectors known ($r^o$ in number)in the inertial frame $O$ (denoted as $p_j$ for $j=1,2,\ldots,r^o$). The idea is to measure the locations of these beacons in the vehicle-fixed frame $S$ (denoted $a_j$) with the help of optical sensors (marked as green in figure \ref{fig:1}). 
 
 It is important to clarify at this stage that at any given discrete-time instant $t_i$, the number of observed beacons and inertial vectors by the vehicle could be varying. We therefore introduce the notations, $r_i^g$ and $r_i^o$ to denote the corresponding inertial and optical measurements available. It should be evident that $r_i^g \in \{1,2,\ldots,r^g\}$ and $r_i^o \in \{1,2,\ldots,r^o\}$. We therefore have $\binom{r_i^o}{2}$ unique relative position vectors, which are the vectors connecting any two of these optical beacon measurements. If two or more optical measurements are available, the number of vector measurements that can be used to estimate attitude are $\binom{r_i^o}{2} + r_i^g$. It has to be noted that attitude of the body can be uniquely computed only if $\binom{r_i^o}{2} + r_i^g \geq 2$, . If at least two inertial vector measurements are available then beacon measurements are not required to estimate attitude, however at least one beacon measurement is necessary for the estimation of relative position. It has been assumed that the velocities of the vehicle can be directly measured.

\subsubsection{Pose measurement model}
Employing the notation from figure \ref{fig:1}, at the time instant $t_i$, we obtain
\begin{equation}\label{eq:7}
    p_j = R(q_j^k + s^k) + b = Ra_j + b,
\end{equation}
in the absence of measurement noise. Here $j \in \{1,2,\ldots r_i^o\}$. In the presence of measurement noise, $a_j^m$ can be written as
\begin{equation*}
    a_j^m = (q_j^k)^m + s^k,
\end{equation*}
Let $\Bar{p} = \frac{1}{r_i^o}\sum_{j=1}^{r_i^o}p_j$ be the mean of vectors $p_j$, and $\Bar{a}^m = \frac{1}{r_i^o}\sum_{j=1}^{r_i^o}a_j^m$ be the mean of vectors $a_j^m$. We obtain the following relation from \eqref{eq:7}:
\begin{equation}\label{eq:9}
    \Bar{a}^m = R\T(\Bar{p}-b) + \zeta,
\end{equation}
where $\zeta$ is the additive
measurement noise obtained by averaging the measurement noise. As stated in the previous sub-section, we obtain $n := \binom{r_i^o}{2}$ relative vectors from optical measurements. They are denoted as $d_j = p_\lambda - p_l$ in $O$ and the corresponding vectors in the frame $S$ are denoted as $l_j = a_\lambda - a_l$ with $\lambda,l\in \{1,2,\ldots,r_i^o\}; \lambda\neq l$. We have
\begin{equation}\label{eq:10}
    d_j = Rl_j \Rightarrow D = RL.
\end{equation}
Putting them in the matrix form as $D = [d_1 \ldots d_n]$ and $L = [l_1 \ldots l_n] \in \bR^{3\times n}$, we obtain
\begin{equation}\label{eq:11}
    L^m = R\T D + \mathscr{L},
\end{equation}
where $\mathscr{L}\in\bR^{3\times n}$
consists of the additive noise in the vector measurements made in the body frame $S$.

\section{Discretization of Rigid body Kinematics}\label{sec:3}

Consider the time interval $[t_0,T]$. Let $\Omega \in \bR^3$ and $\varv\in \bR^3$ be the rotational and translational velocity of the rigid body respectively in frame $S$. $R \in \SO$ is the rotation matrix from body frame to inertial frame and $b \in \bR^3$ is the position of rigid body with respect to frame $O$ expressed in frame $S$. The generalized velocity of the rigid body is constructed as $\xi = [\Omega \;\; \varv]\T $ and the pose of the rigid  body is,
\begin{equation*}
    \SE \ni g = \begin{bmatrix} R & b \\ 0 & 1 \end{bmatrix}.
\end{equation*}

The continuous time rigid body kinematics are:
\begin{equation*}
    \Dot{R} = R\Omega^{\times}, \quad \Dot{b} = R\varv \Rightarrow \Dot{g} = g \xi^{\vee},
\end{equation*}
where $\xi^{\vee} := \begin{bmatrix} \Omega^\times & \varv \\ 0 & 0\end{bmatrix}$. \\

For the discrete-time pose estimation of the rigid body, the continuous time kinematics are discretized as
\begin{equation}\label{eq:14}
    R_{i+1} = R_i\exp{\left(\frac{h}{2}(\Omega_{i+1}+\Omega_i)^\times\right)}, \quad b_{i+1} = b_i + \frac{h}{2}R_{i+1}\left(\varv_{i} + \varv_{i+1}\right).
\end{equation}
Therefore, the discrete-time kinematics of the rigid body pose can be expressed as
\begin{equation}\label{eq:15}
    g_{i+1} = g_i\begin{bmatrix} \exp{\left(\frac{h}{2}(\Omega_{i+1}+\Omega_i)^\times\right)} & \exp{\left(\frac{h}{2}(\Omega_{i+1}+\Omega_i)^\times\right)}\frac{h}{2}\left(\varv_{i} + \varv_{i+1}\right) \\ 0 & 1 \end{bmatrix}.
\end{equation}

\section{Discrete-time estimation of motion from measurements}\label{sec:4}

Consider the time interval $[t_0,T]$. Let $\hat{\Omega}_i$ and $\hat{\varv}_i$ be the estimated rotational and translational velocity of the rigid body respectively in the frame $S$ at time instant $t_i$. $\hat{R}_i$ is the estimated rotation matrix from $S$ to $O$ and $\hat{b}_i$ is the estimated position of rigid body with respect to frame $O$ expressed in frame $S$ at time instant $t_i$. The generalized estimated velocity of the rigid body is constructed as $\hat{\xi}_i := [\hat{\Omega}_i \;\; \hat{\varv}_i]\T $.
From \eqref{eq:15}, the estimated pose and its kinematics can be computed as
\begin{align}
    &  \SE \ni \hat{g}_i := \begin{bmatrix}\hat{R}_i & \hat{b}_i \\ 0 & 1 \end{bmatrix}, \nonumber \\ & \hat{g}_{i+1} = \hat{g}_i\begin{bmatrix} \exp{\left(\frac{h}{2}(\hat{\Omega}_{i+1}+\hat{\Omega}_i)^\times\right)} & \exp{\left(\frac{h}{2}(\hat{\Omega}_{i+1}+\hat{\Omega}_i)^\times\right)}\frac{h}{2}\left(\hat{\varv}_{i} + \hat{\varv}_{i+1}\right) \\ 0 & 1 \end{bmatrix}.
\end{align}

The pose estimation error $h_i$ at time instant $t_i$ can be computed as
\begin{equation}
    \SE \ni \Bar{g}_i := g_i\hat{g}_i^{-1} = \begin{bmatrix} Q_i & b_i - Q_i\hat{b}_i \\ 0 & 1\end{bmatrix} = \begin{bmatrix} Q_i & x_i\\ 0 & 1\end{bmatrix},
\end{equation}
where $Q_i = R_i\hat{R}_i\T $ is the attitude estimation error and $x_i=b_i - Q_i\hat{b}_i$ is the position estimation error. The estimation error in the generalized velocity is denoted as
\begin{equation}\label{eq:13}
    \varphi_i := \varphi(\xi^m_i,\hat{\xi}_i) = \begin{bmatrix} \omega_i \\ v_i \end{bmatrix} = \xi^m_i - \hat{\xi}_i,
\end{equation}
where $\omega_i = \Omega^m_i - \hat{\Omega}_i$ is the angular velocity estimation error and $v_i = \varv^m_i - \hat{\varv}_i$ is the translational velocity estimation error. Here, $\Omega^m_i$ and $\varv^m_i$ denote the measurements of angular and translational velocities respectively at time instant $t_i$. The discrete-time kinematics of the attitude estimation error and the position estimation error are evaluated as
\begin{align}\label{eq:19}
     Q_{i+1} & = R_{i+1}\hat{R}\T_{i+1} \nonumber \\ & = Q_i\hat{R}_i\exp{\left(\frac{h}{2}(\omega_{i+1} + \omega_i)^\times\right)}\hat{R}_i\T,
\end{align}
and 
\begin{align*}
    x_{i+1} & = b_{i+1} - Q_{i+1}\hat{b}_{i+1} \nonumber\\
    & = b_{i} + R_{i+1}\frac{h}{2}\left(\varv_{i} + \varv_{i+1}\right) - Q_{i+1}\left(\hat{b}_{i} + \hat{R}_{i+1}\frac{h}{2}\left(\hat{\varv}_{i} + \hat{\varv}_{i+1}\right)\right) \nonumber \\
    & = b_i - Q_{i+1}\hat{b}_i + R_{i+1}\frac{h}{2}\left(v_{i} + {v}_{i+1}\right) \nonumber \\
    & = b_i - Q_i\hat{R}_i\exp{\left(\frac{h}{2}(\omega_{i+1} + \omega_i)^\times\right)}\hat{R}_i\T\hat{b}_i + R_{i+1}\frac{h}{2}\left(v_{i} + {v}_{i+1}\right).
\end{align*}

Approximating $\exp{\left(\frac{h}{2}(\omega_{i+1} + \omega_i)^\times\right)}$ by the first two terms in the expansion as
\begin{equation}\label{eq:21}
    \exp{\left(\frac{h}{2}(\omega_{i+1} + \omega_i)^\times\right)} \approx I + \frac{h}{2}(\omega_{i+1} + \omega_i)^\times,
\end{equation}
we have,
\begin{align}\label{eq:17}
    x_{i+1} & = b_i - Q_i\hat{b}_i - \frac{h}{2}Q_i\left(\hat{R}_i(\omega_{i+1} + \omega_i)\right)^\times + R_{i+1}\frac{h}{2}\left(v_{i} + {v}_{i+1}\right) \nonumber \\
    & = x_i - \frac{h}{2}Q_i\left(\hat{R}_i(\omega_{i+1} + \omega_i)\right)^\times + R_{i+1}\frac{h}{2}\left(v_{i} + {v}_{i+1}\right).
\end{align}
It has to be noted that approximation in \eqref{eq:21} is accurate for small values of $h$ and may affect the stability results for very high values of $h$.

\subsection{Discrete-time   optimal pose estimator based on Lagrange-d’Alembert principle}

The error in the attitude estimation is encapsulated by Wahba's cost function\cite{wahba1965least}. Thus, the artificial potential function for rotation estimation error is defined as
\begin{equation}\label{eq:18}
    \mathcal{U}^r_i := \mathcal{U}^r(\hat{g}_i,L^m_i,D_i) = \frac{1}{2}k_p\langle D_i-\hat{R}_iL^m_i, (D_i-\hat{R}_iL^m_i)W_i \rangle,
\end{equation}
where $D$ and $L^m$ are as defined in \eqref{eq:10}-\eqref{eq:11}, $W = \text{diag}(w_j) \in \bR^{n\times n}$ is a positive definite diagonal matrix of the weight factors for the measured directions, and $k_p > 0$ is a scalar gain. The artificial potential function for translation estimation error is defined as:
\begin{equation}\label{eq:24}
    \mathcal{U}^t_i := \mathcal{U}^t(\hat{g}_i,\Bar{a}^m_i,\Bar{p}_i) = \kappa \norm{y_i}^2 := \kappa|| \Bar{p}_i - \hat{R}_i\Bar{a}^m_i - \hat{b}_i ||^2,
\end{equation}
where $\Bar{p}$ and $\Bar{a}^m$ are as per \eqref{eq:9} and $\kappa > 0$ is a scalar gain. The total artificial potential energy will be the sum of the artificial rotational and translational potential functions:
\begin{align}\label{eq:20}
    \mathcal{U}_i := \mathcal{U}(\hat{g}_i,L^m_i,D_i,\Bar{a}^m_i,\Bar{p}_i) & = \mathcal{U}^r(\hat{g}_i,L^m_i,D_i) + \mathcal{U}^t(\hat{g}_i,\Bar{a}^m_i,\Bar{p}_i) \nonumber \\
    & = \frac{1}{2}k_p\langle D_i-\hat{R}_iL^m_i, (D_i-\hat{R}_iL^m_i)W_i \rangle \nonumber \\ & \qquad + \kappa|| \Bar{p}_i - \hat{R}_i\Bar{a}^m_i - \hat{b}_i ||^2.
\end{align}
We define the artificial kinetic energy of the system as a quadratic in the generalized velocity estimation error:
\begin{align}
    \mathcal{T}_i & := \mathcal{T}\left(\varphi(\xi^m_i,\xi_i),\varphi(\xi^m_{i+1},\xi_{i+1})\right) \nonumber\\ & = \frac{m}{2}\left(\varphi(\xi^m_i,\xi_i)+\varphi(\xi^m_{i+1},\xi_{i+1})\right)\T\left(\varphi(\xi^m_i,\xi_i)+\varphi(\xi^m_{i+1},\xi_{i+1})\right),
\end{align}
where $m > 0$ is a scalar, and $\varphi(\xi^m_i,\xi_i)$ and $\varphi(\xi^m_{i+1},\xi_{i+1})$ are according to \eqref{eq:13}.

Note that the artificial kinetic energy $\mathcal{T}_i$ can also be written as the summation of the artificial rotational kinetic energy $\mathcal{T}_i^r$ and artificial translational kinetic energy $\mathcal{T}_i^t$ employing $\eqref{eq:13}$ as
\begin{align}
    \mathcal{T}_i & = \mathcal{T}^r(\omega_{i+1},\omega_i) + \mathcal{T}^t(v_{i+1},v_i) = \mathcal{T}^r_i + \mathcal{T}^t_i \nonumber\\
    & = \frac{1}{2}(\omega_{i+1}+\omega_{i})\T m(\omega_{i+1}+\omega_{i}) + \frac{1}{2}(v_{i+1} + v_i)\T m(v_{i+1} + v_i).
\end{align}

Let the discrete-time Lagrangian be defined as the difference between the artificial kinetic energy and artificial potential energy terms:
\begin{align}\label{eq:26}
    \mathscr{L}_i &:= \mathscr{L}(\omega_{i+1},\omega_i,v_{i+1},v_i,\hat{g}_i,L^m_i,D_i,\Bar{a}^m_i,\Bar{p}_i) \nonumber \\ &= \mathcal{T}^r(\omega_{i+1},\omega_i) + \mathcal{T}^t(v_{i+1},v_i) - \mathcal{U}^r(\hat{g}_i,L^m_i,D_i) - \mathcal{U}^t(\hat{g}_i,\Bar{a}^m_i,\Bar{p}_i) \nonumber \\
    & = \frac{1}{2}(\omega_{i+1}+\omega_{i})\T m(\omega_{i+1}+\omega_{i}) + \frac{1}{2}(v_{i+1} + v_i)\T m(v_{i+1} + v_i) \nonumber \\
    & \qquad -\frac{1}{2}k_p\langle D_i-\hat{R}_iL^m_i, (D_i-\hat{R}_iL^m_i)W_i \rangle - \kappa|| \Bar{p}_i - \hat{R}_i\Bar{a}^m_i - \hat{b}_i ||^2.
\end{align}

If the estimation process is started at time $t_0$, then the discrete-time action functional corresponding to the discrete-time Lagrangian \eqref{eq:26} over the time interval $[t_0, T]$ can be expressed as
\begin{align}\label{eq:27}
    \mathfrak{s}_d(\mathscr{L}_i) & := h\sum_{i=0}^N \bigg \{ \frac{1}{2}(\omega_{i+1}+\omega_{i})\T m(\omega_{i+1}+\omega_{i}) + \frac{1}{2}(v_{i+1} + v_i)\T m(v_{i+1} + v_i) \bigg . \nonumber \\
    & \qquad \bigg. -\frac{1}{2}k_p\langle D_i-\hat{R}_iL^m_i, (D_i-\hat{R}_iL^m_i)W_i \rangle - \kappa|| \Bar{p}_i - \hat{R}_i\Bar{a}^m_i - \hat{b}_i ||^2 \bigg\}.
\end{align}

We are now ready to state our first result on optimal pose estimation.
\begin{proposition}\label{prop}
The variational filter for pose minimizing the action functional $\mathfrak{s}_d(\mathscr{L}_i)$ defined in \eqref{eq:27} is given as
\begin{equation}\label{eq:28}
\begin{cases}
\varphi_{i+2} + \varphi_{i+1} = \exp{\left(-\frac{h}{2}(\hat{\Omega}_{i+2}+\hat{\Omega}_{i+1})^\times\right)}\left[(\varphi_{i+1} + \varphi_i) - \frac{h}{2m}Z^\prime_i - \frac{h}{2m}\eta_{i+1} \right] \\
\hat{\xi}_i = \xi_i^m - \varphi_i, \\
\hat{g}_{i+1} = \hat{g}_i\begin{bmatrix} \exp{\left(\frac{h}{2}(\hat{\Omega}_{i+1}+\hat{\Omega}_i)^\times\right)} & \exp{\left(\frac{h}{2}(\hat{\Omega}_{i+1}+\hat{\Omega}_i)^\times\right)}\frac{h}{2}\left(\hat{\varv}_{i} + \hat{\varv}_{i+1}\right) \\ 0 & 1 \end{bmatrix},
\end{cases}
\end{equation}
where $\eta_{i+1}$ contains Rayleigh dissipation terms for angular and translational motions defined as
\begin{equation*}
    \eta_{i+1} := \begin{bmatrix} \tau_{i+1} \\ f_{i+1} \end{bmatrix},
\end{equation*}
with $Z_i^\prime := Z^\prime(\hat{g}_{i+1},\hat{g}_i,L^m_{i+1},D_{i+1},\Bar{a}^m_{i+1},\tau_i,f_i)$ defined by
\begin{equation*}
    Z^\prime_i := \begin{bmatrix}
    -k_pS_{\Gamma_{i+1}}(\hat{R}_{i+1}) + m(\hat{\varv}_{i+1} + \varv_i)^\times(v_{i+1} + v_i) + \kappa(\Bar{a}^m_{i+1})^\times\hat{R}_{i+1}\T y_{i+1} \\
    \kappa\hat{R}_{i+1}\T y_{i+1},
    \end{bmatrix}
\end{equation*}
where $S_{\Gamma_i}(\hat{R}_i) := \text{vex}(\Gamma_i\T\hat{R}_i - \hat{R}_i\T\Gamma_i)$ and $\Gamma_i := D_iW_i(L^m_i)\T$.

\end{proposition}

\begin{proof}
Consider a first variation in the discrete attitude estimate as
\begin{equation}\label{eq:29}
    \delta \hat{R}_i = \hat{R}_i\Sigma_i^{\times},
\end{equation}
where $\Sigma_i \in \bR^3$ represents a variation   for the discrete attitude estimate. For fixed end-point variations, we have $\Sigma_0 = \Sigma_N = 0$. A first order approximation is to assume that $\hat{\Omega}^\times$ and $\delta\hat{\Omega}^\times$ commute. Taking the first variation of the discrete-time attitude kinematics according to the first equation of \eqref{eq:14} and comparing with \eqref{eq:29} we get
\begin{align}\label{eq:30}
    \delta\hat{R}_{i+1} & = \delta\hat{R}_i\exp{\left(\frac{h}{2}(\hat{\Omega}_{i+1} + \hat{\Omega}_i)^\times\right)} \nonumber \\ & + \frac{h}{2}\hat{R}_i\exp{\left( \frac{h}{2}(\hat{\Omega}_{i+1}+\hat{\Omega}_i)^\times \right)}\delta(\hat{\Omega}_{i+1}+\hat{\Omega}_i)^\times \nonumber \\
    & = \hat{R}_{i+1}\Sigma_{i+1}^\times.
\end{align}
The above can be rearranged to
\begin{align}\label{eq:31}
    \hat{R}_{i+1}\frac{h}{2}\delta(\hat{\Omega}_{i+1}+\hat{\Omega}_i)^\times & = \hat{R}_{i+1}\Sigma_{i+1}^\times \nonumber \\ & - \hat{R}_{i+1}\text{Ad}_{\exp{\left(-\frac{h}{2}(\hat{\Omega}_{i+1}+\hat{\Omega}_i)^\times\right)}}\Sigma_i^\times \nonumber \\
    \Rightarrow \frac{h}{2}\delta(\hat{\Omega}_{i+1}+\hat{\Omega}_i)^\times & = \Sigma_{i+1}^\times - \text{Ad}_{\exp{\left(-\frac{h}{2}(\hat{\Omega}_{i+1}+\hat{\Omega}_i)^\times\right)}}\Sigma_i^\times,
\end{align}
which can be equivalently written as an equation in $\bR^3$ as follows:
\begin{equation}\label{eq:35}
    \frac{h}{2}\delta(\hat{\Omega}_{i+1}+\hat{\Omega}_i) = \Sigma_{i+1} - \exp{\left(-\frac{h}{2}(\hat{\Omega}_{i+1}+\hat{\Omega}_i)^\times\right)}\Sigma_i.
\end{equation}
Taking the variation of $\omega_i = \Omega_i^m - \hat{\Omega}_i$
\begin{equation}\label{eq:36}
    \delta(\omega_{i+1}+\omega_i) = - \delta(\hat{\Omega}_{i+1}+\hat{\Omega}_i).
\end{equation}
Next, we assume the variation in $\hat{b}_i$ to be,
\begin{equation}\label{eq:37}
    \delta\hat{b}_i = \hat{R}_i\rho_i,
\end{equation}
where $\rho_i \in \bR^3$ represents the variation in the discrete position estimate. For fixed end-point variations, we have $\rho_0 = \rho_N = 0$. Taking the first variation of the discrete-time position kinematics according to the second equation of \eqref{eq:14} and comparing with \eqref{eq:37} we get
\begin{align}\label{eq:38}
    &\delta\hat{b}_{i+1} = \delta\hat{b}_i + \delta\hat{R}_{i+1}\frac{h}{2}(\hat{\varv}_{i+1} + \hat{\varv}_i) + \hat{R}_{i+1}\frac{h}{2}\delta(\hat{\varv}_{i+1} + \hat{\varv}_i) \nonumber\\
    &\Rightarrow \hat{R}_{i+1}\rho_{i+1} = \hat{R}_i\rho_i + \hat{R}_{i+1}\Sigma_{i+1}^\times\frac{h}{2}(\hat{\varv}_{i+1} + \hat{\varv}_i) + \hat{R}_{i+1}\frac{h}{2}\delta(\hat{\varv}_{i+1} + \hat{\varv}_i) \nonumber \\
    &\Rightarrow \frac{h}{2}\delta(\hat{\varv}_{i+1} + \hat{\varv}_i) = \rho_{i+1} - \exp{\left(-\frac{h}{2}(\hat{\Omega}_{i+1}+\hat{\Omega}_i)^\times\right)}\rho_i - \Sigma_{i+1}^\times\frac{h}{2}(\hat{\varv}_{i+1} + \hat{\varv}_i)
\end{align}
and the variation of $v_i = \varv_i^m - \hat{\varv}_i$ gives us
\begin{equation}\label{eq:39}
    \delta(v_{i+1}+v_i) = - \delta(\hat{\varv}_{i+1}+\hat{\varv}_i).
\end{equation}

We have $y_i = \Bar{p}_i - \hat{R}_i\Bar{a}_i^m - \hat{b}_i$. Therefore,
\begin{align}\label{eq:40}
    \delta y_i & = -\delta\hat{R}_i\Bar{a}_i^m -\delta\hat{b}_i \nonumber \\
    & = \hat{R}_i\Sigma_i^\times\Bar{a}_i^m - \hat{R}_i\rho_i \nonumber \\
    & = \hat{R}_i\left(\left(\Bar{a}_i^m\right)^\times\Sigma_i - \rho_i \right).
\end{align}

Consider the artificial potential energy term in \eqref{eq:18}. Taking its first variation with respect to the estimated attitude $\hat{R}$, we get
\begin{align}\label{eq:41}
    \delta\mathcal{U}^r_i & = \frac{k_p}{2}\left\{ \langle -\delta\hat{R}_iL_i^m , (D_i-\hat{R}_iL_i^m)W_i \rangle \right. \left. + \langle D_i-\hat{R}_iL_i^m , (-\delta\hat{R}_iL_i^m)W_i \rangle \right\} \nonumber \\
    & = k_p\langle -\delta\hat{R}_iL_i^m , (D_i-\hat{R}_iL_i^m)W_i \rangle \nonumber \\
    & = k_p\langle -\hat{R}_i\Sigma_i^\times, (D_i-\hat{R}_iL_i^m)W_i \rangle \nonumber \\
    & = k_p\text{tr}\left( (L_i^m)\T\Sigma_i^\times\hat{R}_i\T(D_i-\hat{R}_iL_i^m)W_i \right) \nonumber \\ & = k_p\text{tr}\left((\Sigma_i^\times)\T L_i^mW_iD_i\T\hat{R}_i\right) \nonumber \\
    &= k_p\langle \Sigma_i^\times , L_i^mW_iD_i\T\hat{R}_i\rangle \nonumber \\
    &= k_p\frac{1}{2} \langle\, \Sigma^\times , L_i^mW_iD_i\T\hat{R}_i - \hat{R}_i\T D_iW_i(L_i^m)\T\rangle \nonumber \\
    &= k_p\frac{1}{2}\langle\, \Sigma_i^\times , \Gamma_i\T\hat{R}_i - \hat{R}_i\T \Gamma_i \rangle = k_pS_{\Gamma_i}\T(\hat{R}_i)\Sigma_i.
\end{align}
Similarly, taking the first variation of the artificial potential energy function in \eqref{eq:24} and using results from \eqref{eq:40}, we get
\begin{equation}\label{eq:42}
    \delta\mathcal{U}_i^t = \kappa y_i\T\delta y_i = \kappa y_i\T\left(\left(\Bar{a}_i^m\right)^\times\Sigma_i - \rho_i \right).
\end{equation}
Similarly, we also obtain the variations in artificial kinetic energies as follows
\begin{equation}\label{eq:43}
    \delta\mathcal{T}^r_i = \frac{2m}{h}(\omega_{i+1} + \omega_i)\T\left(\exp{\left(-\frac{h}{2}(\hat{\Omega}_{i+1}+\hat{\Omega}_i)^\times\right)}\Sigma_i - \Sigma_{i+1} \right),
\end{equation}
\begin{equation}\label{eq:44}
    \delta\mathcal{T}^t_i = \frac{2m}{h}(v_{i+1} + v_i)\T\left(\exp{\left(-\frac{h}{2}(\hat{\Omega}_{i+1}+\hat{\Omega}_i)^\times\right)}\rho_i + \Sigma_{i+1}^\times\frac{h}{2}(\hat{\varv}_{i+1} + \hat{\varv}_i) - \rho_{i+1} \right),
\end{equation}
with the help of relations from \eqref{eq:35}-\eqref{eq:39}. Taking the first variation of the discrete-time action sum in \eqref{eq:27} and employing \eqref{eq:41}-\eqref{eq:44} we obtain
\begin{align}
    & \delta\mathfrak{s}_d = h\sum_{i=0}^N\left\{ \delta\mathcal{T}_i^r + \delta\mathcal{T}_i^t - \delta\mathcal{U}_i^r - \delta\mathcal{U}_i^t \right\} \nonumber \\
    & = h\sum_{i=0}^N \left\{ \frac{2m}{h}(\omega_{i+1} + \omega_i)\T\left(\exp{\left(-\frac{h}{2}(\hat{\Omega}_{i+1}+\hat{\Omega}_i)^\times\right)}\Sigma_i - \Sigma_{i+1} \right) \right. \nonumber \\
    & + \frac{2m}{h}(v_{i+1} + v_i)\T\left(\exp{\left(-\frac{h}{2}(\hat{\Omega}_{i+1}+\hat{\Omega}_i)^\times\right)}\rho_i + \Sigma_{i+1}^\times\frac{h}{2}(\hat{\varv}_{i+1} + \hat{\varv}_i) - \rho_{i+1} \right) \nonumber \\
    & \qquad \bigg. -k_pS_{\Gamma_i}\T(\hat{R}_i)\Sigma_i -\kappa y_i\T\left(\left(\Bar{a}_i^m\right)^\times\Sigma_i - \rho_i \right) \bigg \}.
\end{align}

We now apply the discrete Lagrange-d'Alembert principle \cite{marsden2001discrete} with two Rayleigh dissipation terms $\tau_i \in \bR^3$ and $f_i \in \bR^3$ for angular and translational motion respectively,
\begin{align}\label{eq:45}
    & \delta\mathfrak{s}_d + h\sum_{i=0}^{N-1}\left\{ \tau_i\T\Sigma_i + f_i\T\rho_i \right\} = 0 \nonumber \\ & \Rightarrow \sum_{i=0}^{N-1} \left\{ 2m(\omega_{i+1} + \omega_i)\T\left(\exp{\left(-\frac{h}{2}(\hat{\Omega}_{i+1}+\hat{\Omega}_i)^\times\right)}\Sigma_i - \Sigma_{i+1} \right) \right. \nonumber \\
    & + 2m(v_{i+1} + v_i)\T\left(\exp{\left(-\frac{h}{2}(\hat{\Omega}_{i+1}+\hat{\Omega}_i)^\times\right)}\rho_i + \Sigma_{i+1}^\times\frac{h}{2}(\hat{\varv}_{i+1} + \hat{\varv}_i) - \rho_{i+1} \right) \nonumber \\
    & \qquad \bigg. -k_phS_{\Gamma_i}\T(\hat{R}_i)\Sigma_i -\kappa h y_i\T\left(\left(\Bar{a}_i^m\right)^\times\Sigma_i - \rho_i \right) h\tau_i\T\Sigma_i + h f_i\T\rho_i \bigg \} = 0.
\end{align}
 
Splitting \eqref{eq:45} into two equations assuming independence of $\Sigma_i$ and $\rho_i$ will give us
\begin{align}\label{eq:46}
& 2m(\omega_{i+2} + \omega_{i+1})\T\exp{\left(-\frac{h}{2}(\hat{\Omega}_{i+2}+\hat{\Omega}_{i+1})^\times\right)} - 2m(\omega_{i+1} + \omega_{i})\T \nonumber \\
& \qquad - hm(v_{i+1} + v_i)\T(\hat{\varv}_{i+1} + \hat{\varv}_i)^\times - -k_phS_{\Gamma_{i+1}}\T(\hat{R}_{i+1}) \nonumber \\
& \qquad -\kappa h y_{i+1}\T\left(\Bar{a}_{i+1}^m\right)^\times + h\tau_{i+1}\T = 0,
\end{align}
and
\begin{align}\label{eq:47}
& 2m(v_{i+2} + v_{i+1})\T\exp{\left(-\frac{h}{2}(\hat{\Omega}_{i+2}+\hat{\Omega}_{i+1})^\times\right)} - 2m(v_{i+1} + v_i)\T \nonumber \\
& \qquad +\kappa h y_{i+1}\T\hat{R}_{i+1} + hf_i\T = 0.
\end{align}
The above equations can be simplified to obtain
\begin{align}\label{eq:48}
\omega_{i+2} + \omega_{i+1} & = \exp{\left(-\frac{h}{2}(\hat{\Omega}_{i+2}+\hat{\Omega}_{i+1})^\times\right)} \bigg[ \omega_{i+1} + \omega_{i} + \frac{h}{2m}\big\{ k_pS_{\Gamma_{i+1}}(\hat{R}_{i+1}) \big. \bigg. \nonumber \\
& \bigg. \big. -m(\hat{\varv}_{i+1} + \varv_i)^\times(v_{i+1} + v_i) - \kappa(\Bar{a}^m_{i+1})^\times\hat{R}_{i+1}\T y_{i+1} - \tau_{i+1} \big\}  \bigg] = 0,
\end{align}
and
\begin{align}\label{eq:49}
v_{i+2} + v_{i+1} & = \exp{\left(-\frac{h}{2}(\hat{\Omega}_{i+2}+\hat{\Omega}_{i+1})^\times\right)} \bigg[ v_{i+1} + v_{i} \bigg. \nonumber \\
& \qquad \bigg. -\frac{h}{2m}\left\{ \kappa\hat{R}_{i+1}\T y_{i+1} + f_{i+1} \right\} \bigg] = 0,
\end{align}
which in turn can be combined to obtain \eqref{eq:28}. \qed
\end{proof}

\section{Discrete-time asymptotically stable and optimal pose estimator}\label{sec:5}

For the discrete-time Lyapunov analysis, we use the same combination of artificial potential energy like terms as defined in \eqref{eq:20}. We construct a new kinetic energy like term to encapsulate the error in the generalized velocity estimation. This term will aid us in the Lyapunov analysis. We propose the following kinetic energy like term:
\begin{equation}
    \mathcal{T}^l_i := \mathcal{T}^l\left(\varphi(\xi^m_i,\xi_i)\right) = \frac{m}{2}\varphi(\xi^m_i,\xi_i)\T\varphi(\xi^m_i,\xi_i),
\end{equation}
where $m > 0$ is a scalar. We carry out the Lyapunov analysis in the absence of measurement errors. The following Lemma provides the form of the artificial potential energy term in the absence of measurement errors.
\begin{lemma}
In the absence of measurement noise, the artificial potential energy defined in \eqref{eq:20} takes the form
\begin{equation}
    \mathcal{U}_i = \mathcal{U}(\hat{g}_i,L^m_i,D_i,\Bar{a}^m_i,\Bar{p}_i) = k_p\langle I-Q_i,K_i \rangle +\kappa y_i\T y_i,
\end{equation}
where $K_i = D_iW_iD_i\T$ and $y_i = y(h_i,\Bar{p}_i) = Q_i\T x_i + (I-Q_i\T)\Bar{p}_i$.
\end{lemma}

\begin{proof}
 In the absence of measurements errors, we have $L^m_i =L_i$, $\Bar{a}^m_i = \Bar{a}_i$ and $\xi^m_i = \xi_i$. The rotational potential function can be rewritten as
 \begin{align*}
     \mathcal{U}^r(\hat{g}_i,L^m_i,D_i) & = \frac{k_p}{2} \langle D_i-\hat{R}_iL^m_i, (D_i-\hat{R}_iL^m_i)W_i \rangle \nonumber \\
     & = \frac{k_p}{2}\langle\,I - R_i\hat{R}_i\T,D_iW_iD_i\T\rangle \nonumber \\
     \Rightarrow \mathcal{U}^r(Q_i) & =  k_p\langle\,I-Q_i,K_i\rangle \;\; \text{where} \;\; K_i = D_iW_iD_i\T.
\end{align*}
 Since we have $\hat{R}_iL_i = Q_i\T D_i$ and $\hat{b}_i = Q_i\T(b_i-x_i)$ we get
 \begin{equation*}
     y_i = y(h_i,\Bar{p}_i) = \Bar{p}_i - \hat{R}_i\Bar{a}_i - \hat{b}_i = Q_i\T x_i + (I-Q_i\T)\Bar{p}_i,
 \end{equation*}
 as the form of potential energy in the absence of measurement errors. \qed
\end{proof}
We state the following assumption that is relevant for the proof of asymptotic stability.
\begin{assumption}\label{assumption}
The measured beacons and inertial vectors are fixed throughout the estimation process which results in $K_i = K$ and $\Bar{p}_i = \Bar{p}$ for some constants $K \in \bR^{n\times n}$ and $\Bar{p} \in \bR^3$ for all $t_i$.
\end{assumption}
We are now ready to state the main result of this article on an optimal asymptotically stable pose filter.
\begin{theorem}
Consider the following form of the dissipation term in \eqref{eq:52}
\begin{align}\label{eq:52}
    \eta_{i+1} & = \frac{2}{h} \Bigg\{ m(\varphi_{i+1} + \varphi_i) - \frac{h}{2}Z^\prime_i \Bigg. \nonumber \\ & \qquad \Bigg. -\frac{m}{m+l}\exp{\left(\frac{h}{2}(\hat{\Omega}_{i+2}+\hat{\Omega}_{i+1})^\times\right)}\left[2m\varphi_{i+1} - hZ_{i+1}\right] \Bigg\}.
\end{align}
Then the resulting nonlinear pose estimator given by,
\begin{equation}\label{eq:53}
\begin{cases}
\varphi_{i+1} = \frac{1}{m+l}\left[ (m-l)\varphi_i - hZ_i \right] \\
\hat{\xi}_i = \xi^m_i - \varphi_i, \\
\hat{g}_{i+1} = \hat{g}_i\begin{bmatrix} \exp{\left(\frac{h}{2}(\hat{\Omega}_{i+1}+\hat{\Omega}_i)^\times\right)} & \exp{\left(\frac{h}{2}(\hat{\Omega}_{i+1}+\hat{\Omega}_i)^\times\right)}\frac{h}{2}\left(\hat{\varv}_{i} + \hat{\varv}_{i+1}\right) \\ 0 & 1 \end{bmatrix},
\end{cases}
\end{equation}
where and $Z_i = Z(\hat{g}_i,L^m_i,D_i,\Bar{a}^m_i,\Bar{p}_i)$ is defined by
\begin{equation}\label{eq:54}
    Z(\hat{g}_i,L^m_i,D_i,\Bar{a}^m_i,\Bar{p}_i) = \begin{bmatrix}
    -k_pS_{\Gamma_i}(\hat{R}_i) + \kappa\hat{R}_i\T\left(Q_i\T(\Bar{p} - b_i)\right)^\times(y_{i+1} + y_i) \\
    \kappa\hat{R}_{i+1}\T(y_{i+1} + y_i) 
    \end{bmatrix}
\end{equation}
where $\Gamma_i = D_iW_i(L^m_i)\T$, $S_{\Gamma_i}(\hat{R}_i) = \text{vex}(\Gamma_i\T\hat{R}_i - \hat{R}_i\T\Gamma_i)$
and $l>0,l\neq m$, is asymptotically stable under Assumption \ref{assumption} at the estimation error state $(\Bar{g}_i, \varphi_i) =
(I, 0)$. Further, the domain of attraction of $(\Bar{g}_i, \varphi_i) = (I, 0)$ is a dense
open subset of $\SE\times \bR^6$.
\end{theorem}

\begin{proof}

We choose the following discrete-time Lyapunov candidate:
\begin{equation*}
    V_i := V(Q_i,h_i,\Bar{p}_i) := \mathcal{U}_i + \mathcal{T}_i,
\end{equation*}

The stability of the attitude and angular velocity error can be shown by analyzing $\Delta V_i = \Delta\mathcal{U}_i + \Delta\mathcal{T}_i$.\\

Using Assumption \ref{assumption}, we first calculate
\begin{align*}
     & \Delta\mathcal{U}^r_i = \mathcal{U}^r_{i+1} - \mathcal{U}^r_i = k_p\langle\,I-Q_{i+1},K\rangle - k_p\langle\,I-Q_i,K\rangle, \nonumber \\
     & \Delta\mathcal{U}^r_i = k_p\langle\,Q_i - Q_{i+1},K\rangle = -k_p\langle\,\Delta Q_i,K\rangle,
\end{align*}
where, $\Delta Q_i = Q_{i+1} - Q_i$. Now,
\begin{align*}
    \Delta Q_i & = Q_{i+1} - Q_i \nonumber \\
    & = Q_i\left[\hat{R}_i\exp{\left(\frac{h}{2}(\hat{\omega}_{i+1} + \hat{\omega}_i)^\times\right)}\hat{R}_i\T - I\right].
\end{align*}

Approximating $\exp{\left(\frac{h}{2}(\hat{\omega}_{i+1} + \hat{\omega}_i)^\times\right)}$ as shown in \eqref{eq:21}, we have
\begin{align*}
    \Delta Q_i 
    & = Q_i\left[\hat{R}_i\left(I + \frac{h}{2}(\hat{\omega}_{i+1} + \hat{\omega}_i)^\times\right)\hat{R}_i\T - I\right] \nonumber \\
    & = \frac{h}{2}Q_i\left(\hat{R}_i(\hat{\omega}_{i+1} + \hat{\omega}_i)^\times\hat{R}_i\T\right) \nonumber \\
    & = \frac{h}{2}Q_i\left(\hat{R}_i(\hat{\omega}_{i+1} + \hat{\omega}_i)\right)^\times .
\end{align*}

In the absence of measurement errors, we have $L_i^m = R_i\T D_i$.
Therefore,
\begin{align*}
    \Delta\mathcal{U}^r_i & = -\frac{k_ph}{2}\left\langle\,Q_i\left( \hat{R}_i \left(\omega_{i+1} + \omega_i \right) \right)^\times,K_i\right\rangle \nonumber \\
    & = -\frac{k_ph}{2}\left \langle\,R_i(\omega_{i+1} + \omega_i)^\times\hat{R}_i\T,D_iW_iD_i\T\right\rangle \nonumber\\
    & = -\frac{k_ph}{2}\left \langle\,(\omega_{i+1} + \omega_i)^\times\hat{R}_i\T,R_i\T D_iW_iD_i\T\right\rangle \nonumber\\
    & = -\frac{k_ph}{2}\left \langle\,(\omega_{i+1} + \omega_i)^\times\hat{R}_i\T,L_i^mW_iD_i\T\right\rangle.
\end{align*}

We can further simplify the above expression using $\Gamma_i = D_iW_i(L_i^m)\T$.
\begin{align}\label{eq:62}
    \Delta\mathcal{U}^r_i & = -\frac{k_ph}{2}\left \langle\,(\omega_{i+1} + \omega_i)^\times,\Gamma_i\T\hat{R}_i\right\rangle \nonumber\\
    & = -\frac{k_ph}{4}\left \langle\,(\omega_{i+1} + \omega_i)^\times,\Gamma_i\T\hat{R}_i - \hat{R}_i\T\Gamma_i\right\rangle \nonumber\\
    & = -\frac{k_ph}{2}(\omega_{i+1}+\omega_i)\T S_{\Gamma_i}(\hat{R}_i) \nonumber \\
    & = -\frac{h}{2}(\varphi_{i+1} + \varphi_i)\T\begin{bmatrix}
    k_pS_{\Gamma_i}(\hat{R}_i) \\ 0
    \end{bmatrix}
\end{align}
where, $S_{\Gamma_i}(\hat{R}_i) = \text{vex}(\Gamma_i\T\hat{R}_i - \hat{R}_i\T\Gamma_i)$. \\

Similarly we can compute the change in the translational potential energy as follows:
\begin{align*}
        \Delta\mathcal{U}^t_i & = \mathcal{U}^t(y_{i+1}) - \mathcal{U}^t(y_i) \nonumber \\
        & = (y_{i+1} - y_i)\T\kappa(y_{i+1} + y_i)
        \nonumber\\
      \Delta\mathcal{U}^t_i & = (y_{i+1} - y_i)\T\kappa(y_{i+1} + y_i).
\end{align*}

Under Assumption \ref{assumption} we have,
\begin{align}\label{eq:65}
    & y_i = Q_i\T x_i + (I - Q_i\T)\Bar{p} \nonumber \\
    & y_{i+1} = Q_{i+1}\T x_{i+1} +(I - Q_{i+1}\T)\Bar{p}.
\end{align}
Therefore the discrete-time evolution of $y_i$ is
\begin{align}\label{eq:66}
    y_{i+1} - y_i &= Q_{i+1}\T b_{i+1} - \hat{b}_{i+1} - \left( Q_i\T b_i - \hat{b}_i \right) + (Q_i - Q_{i+1})\T\Bar{p} \nonumber \\
    & = \hat{R}_{i+1}R_{i+1}\T\left(b_{i} + R_{i+1}\frac{h}{2}\left(\varv_{i} + \varv_{i+1}\right)\right) - Q_i\T b_i \nonumber \\ & \qquad - \left(\hat{b}_{i} + \hat{R}_{i+1}\frac{h}{2}\left(\hat{\varv}_{i} + \hat{\varv}_{i+1}\right) - \hat{b}_i\right) \nonumber \\
    & = (Q_{i+1} - Q_i)\T(b_i - \Bar{p}) + \hat{R}_{i+1}\frac{h}{2}\left(v_{i} + v_{i+1}\right) \nonumber \\
    & =\frac{h}{2}\left[Q_i\left(\hat{R}_i(\omega_{i+1} + \omega_i)\right)^\times\right]\T(b_i - \Bar{p}) + \hat{R}_{i+1}\frac{h}{2}\left(v_{i} + v_{i+1}\right) \nonumber \\
    & = -\frac{h}{2}\left(\hat{R}_i(\omega_{i+1} + \omega_i)\right)^\times Q_i\T(b_i - \Bar{p}) + \hat{R}_{i+1}\frac{h}{2}\left(v_{i} + v_{i+1}\right) \nonumber \\
    & = \frac{h}{2}\left(Q_i\T(b_i - \Bar{p})\right)^\times\hat{R}_i(\omega_{i+1} + \omega_i) + \hat{R}_{i+1}\frac{h}{2}\left(v_{i} + v_{i+1}\right).
\end{align}

Using \eqref{eq:66}, we obtain
\begin{align}\label{eq:67}
    \Delta\mathcal{U}^t_i & = (y_{i+1} - y_i)\T\kappa(y_{i+1} + y_i) \nonumber \\
    & = \frac{\kappa h}{2}\left[(\omega_{i+1} + \omega_i)\T\hat{R}_i\T\left(Q_i\T(\Bar{p} - b_i)\right)^\times + \left(v_{i} + v_{i+1}\right)\T\hat{R}_{i+1}\T\right](y_{i+1} + y_i) \nonumber \\
    & = \frac{h}{2}\begin{bmatrix}
    \omega_{i+1} + \omega_i \\ v_{i} + v_{i+1}
    \end{bmatrix}\T\begin{bmatrix}
    \kappa\hat{R}_i\T\left(Q_i\T(\Bar{p} - b_i)\right)^\times(y_{i+1} + y_i) \\
    \kappa\hat{R}_{i+1}\T(y_{i+1} + y_i)
    \end{bmatrix} \nonumber \\
    & = \frac{h}{2}(\varphi_{i+1} + \varphi_i)\T\begin{bmatrix}
    \kappa\hat{R}_i\T\left(Q_i\T(\Bar{p} - b_i)\right)^\times(y_{i+1} + y_i) \\
    \kappa\hat{R}_{i+1}\T(y_{i+1} + y_i)
    \end{bmatrix}.
\end{align}

Similarly we can compute the change in the kinetic energy as follows:
\begin{align}\label{eq:68}
        \Delta\mathcal{T}_i & = \mathcal{T}(\varphi_{i+1}) - \mathcal{T}(\varphi_i) \nonumber \\
        & = (\varphi_{i+1} + \varphi_i)\T\frac{m}{2}(\varphi_{i+1} - \varphi_i).
\end{align}
Using values from \eqref{eq:62}, \eqref{eq:67}, and \eqref{eq:68}, we obtain
\begin{align}
    \Delta V_i & = \Delta\mathcal{U}^r_i + \Delta\mathcal{U}^t_i + \Delta\mathcal{T}_i \nonumber\\
    & = -\frac{h}{2}(\varphi_{i+1} + \varphi_i)\T\begin{bmatrix}
    k_pS_{\Gamma_i}(\hat{R}_i) \\ 0
    \end{bmatrix} +  (\varphi_{i+1} + \varphi_i)\T\frac{m}{2}(\varphi_{i+1} - \varphi_i) \nonumber \\ & \qquad + \frac{h}{2}(\varphi_{i+1} + \varphi_i)\T\begin{bmatrix}
    \kappa\hat{R}_i\T\left(Q_i\T(\Bar{p} - b_i)\right)^\times(y_{i+1} + y_i) \\
    \kappa\hat{R}_{i+1}\T(y_{i+1} + y_i) 
    \end{bmatrix} \nonumber \\
    & = \frac{1}{2}(\varphi_{i+1} + \varphi_i)\T\Bigg\{ m(\varphi_{i+1} - \varphi_i) \Bigg. \nonumber \\ & \qquad 
    \left.+ h\begin{bmatrix}
    -k_pS_{\Gamma_i}(\hat{R}_i) + \kappa\hat{R}_i\T\left(Q_i\T(\Bar{p} - b_i)\right)^\times(y_{i+1} + y_i) \\
    \kappa\hat{R}_{i+1}\T(y_{i+1} + y_i) 
    \end{bmatrix} \right\}.
\end{align}
Using \eqref{eq:54} yields
\begin{equation*}
    \Delta V_i = \frac{1}{2}(\varphi_{i+1} + \varphi_i)\T\left\{ m(\varphi_{i+1} - \varphi_i) + hZ_i \right\}.
\end{equation*}

Therefore, $\Delta V_{i+1}$ can be written as
\begin{equation*}
    \Delta V_{i+1} = \frac{1}{2}(\varphi_{i+2} + \varphi_{i+1})\T\left\{ m(\varphi_{i+2} - \varphi_{i+1}) + hZ_{i+1} \right\}.
\end{equation*}
Substituting for $\varphi_{i+2}$ from \eqref{eq:28},
\begin{align*}
    \Delta V_{i+1} &= \frac{1}{2}(\varphi_{i+2} + \varphi_{i+1})\T\Bigg\{ -2m\varphi_{i+1} + hZ_{i+1} \Bigg. \nonumber \\
    & + \Big. \exp{\left(-\frac{h}{2}(\hat{\Omega}_{i+2}+\hat{\Omega}_{i+1})^\times\right)}\left[m(\varphi_{i+1} + \varphi_i) - \frac{h}{2}Z^\prime_i - \frac{h}{2}\eta_{i+1} \right] \Bigg\}.
\end{align*}

Now, in order for $\Delta V_{i+1}$ to be negative definite, we require 
\begin{align}\label{eq:72}
    & \exp{\left(-\frac{h}{2}(\hat{\Omega}_{i+2}+\hat{\Omega}_{i+1})^\times\right)}\left[m(\varphi_{i+1} + \varphi_i) - \frac{h}{2}Z^\prime_i - \frac{h}{2}\eta_{i+1} \right] \nonumber \\
    & \qquad -2m\varphi_{i+1} + hZ_{i+1} = -l(\varphi_{i+2} + \varphi_{i+1}),
\end{align}
where $l > 0, l \neq m$. $\Delta V_{i+1}$ simplifies to 
\begin{equation}\label{eq:75}
    \Delta V_{i+1} = -\frac{l}{2}\left(\varphi_{i+2} + \varphi_{i+1}\right)\T\left(\varphi_{i+2} + \varphi_{i+1}\right).
\end{equation}
Again substituting the value of $\varphi_{i+2}$ from \eqref{eq:28} into \eqref{eq:72}, we have
\begin{align*}
    & \exp{\left(-\frac{h}{2}(\hat{\Omega}_{i+2}+\hat{\Omega}_{i+1})^\times\right)}\left[m(\varphi_{i+1} + \varphi_i) - \frac{h}{2}Z^\prime_i - \frac{h}{2}\eta_{i+1} \right] -2m\varphi_{i+1} \nonumber \\
    & + hZ_{i+1} = -\frac{l}{m} \exp{\left(-\frac{h}{2}(\hat{\Omega}_{i+2}+\hat{\Omega}_{i+1})^\times\right)}\left[m(\varphi_{i+1} + \varphi_i) - \frac{h}{2}Z^\prime_i - \frac{h}{2}\eta_{i+1} \right].
\end{align*}
Streamlining the terms in the above equation, we obtain
\begin{align*}
    & \frac{m+l}{m}\exp{-\left(\frac{h}{2}(\hat{\Omega}_{i+2}+\hat{\Omega}_{i+1})^\times\right)}\left[m(\varphi_{i+1} + \varphi_i) - \frac{h}{2}Z^\prime_i - \frac{h}{2}\eta_{i+1} \right] \nonumber \\
    & \qquad = 2m\varphi_{i+1} - hZ_{i+1},
\end{align*}
which upon further simplification yields
\begin{align*}
    & m(\varphi_{i+1} + \varphi_i) - \frac{h}{2}Z^\prime_i - \frac{h}{2}\eta_{i+1} \nonumber \\
    & \qquad = \frac{m}{m+l}\exp{\left(\frac{h}{2}(\hat{\Omega}_{i+2}+\hat{\Omega}_{i+1})^\times\right)}\left[2m\varphi_{i+1} - hZ_{i+1}\right].
\end{align*}
Rearranging the terms above, we obtain  the value of $\eta_i$ as shown in \eqref{eq:52}. After substituting for $\eta_i$ in \eqref{eq:53}, we get
\begin{equation}\label{eq:79}
    \varphi_{i+2} = \frac{1}{m+l}\left[ (m-l)\varphi_{i+2} - hZ_{i+1} \right].
\end{equation}
\eqref{eq:79} can be rewritten in the previous time step as
\begin{equation}
    \varphi_{i+1} = \frac{1}{m+l}\left[ (m-l)\varphi_i - hZ_i \right].
\end{equation}
We can rewrite $\Delta V_i$ with the help of \eqref{eq:75} to be,
\begin{equation}
    \Delta V_{i} = -\frac{l}{2}\left(\varphi_{i+1} + \varphi_{i}\right)\T\left(\varphi_{i+1} + \varphi_{i}\right).
\end{equation}

We employ the discrete-time La-Salle invariance principle from \cite{lasalle1976stability} considering our domain ($\SE\times\bR^6$) to be a subset of $\bR^{12}$. We use Theorem 6.3 and Theorem 7.9  from Chapter-1 of \cite{lasalle1976stability}. For this we first compute,
\begin{equation}
    \mathscr{E} = \Delta V^{-1}_i(0) = \{(\Bar{g}_i,\varphi_i) \in \SE\times\bR^6\;|\; \varphi_{i+1} + \varphi_i \equiv 0\}.
\end{equation}
Now,
\begin{align}\label{eq:84}
    \varphi_{i+1} + \varphi_i = 0 \Rightarrow \omega_{i+1} + \omega_i = 0, \;\; v_{i+1} + v_{i} = 0
\end{align}

From \eqref{eq:19}, $\omega_{i+1} + \omega_i =0 \Rightarrow Q_{i+1} = Q_i$. Therefore, we have  $\Delta\mathcal{U} = 0$ whenever $\omega_{i+1} + \omega_i = 0$. This implies that the potential function, which is a Morse function, is not changing and therefore has converged to one of its stationary points. At the stationary points we have that $S_{\Gamma_i}(\hat{R}_i) = 0$. Furthermore, $I$ is the global minima of the Morse function with an almost global domain of attraction.\\

From \eqref{eq:17}, we obtain that $x_{i+1} = x_i$ and therefore from \eqref{eq:65} we also have that $y_{i+1}=y_i$. Note that we have $\Bar{g}_{i+1} = \Bar{g}_{i}$ when $Q_{i+1} = Q_{i}$ and $x_{i+1} = x_{i}$. Further substitutiting \eqref{eq:84} into \eqref{eq:53}, we have
\begin{equation*}
    \varphi_i = \frac{hZ_i}{2m} = \frac{h\kappa}{m}\begin{bmatrix}
     \hat{R}_i\T\left(Q_i\T(\Bar{p} - b_i)\right)^\times(y_i) \\
    \hat{R}_{i+1}\T(y_i) 
    \end{bmatrix}.
\end{equation*}

Similarly,
\begin{equation*}
    \varphi_{i+1} = \frac{h\kappa}{m}\begin{bmatrix}
     \hat{R}_{i+1}\T\left(Q_i\T(\Bar{p} - b_{i+1})\right)^\times(y_i) \\
    \hat{R}_{i+2}\T(y_i) 
    \end{bmatrix},
\end{equation*}
and since $\varphi_{i+1} + \varphi_i = 0$, we have $(\hat{R}_{i+1} + \hat{R}_{i+2})\T y_i = 0$ which can be rewritten as $(I + \hat{R}_{i+1}\hat{R}_{i+2}\T)y_i = 0$. It is evident that $(I + \hat{R}_{i+1}\hat{R}_{i+2}\T)$ is non-singular and therefore $y_i = 0$. Note that, $y_i = Q_i\T x_i +  (I-Q_i\T)\Bar{p}$. Therefore, when $Q_i = I$, we have $y_i = x_i$ and subsequently $x_i = 0$ with $Q_i = I$ gives us $\hat{b}_i = b_i$. Therefore the largest invariant set for the estimation error dynamics will be $\mathscr{M} = \{(\Bar{g}_i,\varphi_i) \in \SE\times\bR^6\;|\; \Bar{g}_i = I, \varphi_i = 0 \}$. Furthermore, we obtain the positive limit set as the set,
\begin{align*}
    \mathscr{I} & := \mathscr{M} \cap V_i^{-1}(0) \nonumber \\ & = \{(\Bar{g}_i,\varphi_i) \in \SE\times\bR^6\;|\; \Bar{g}_i = I, \varphi_i = 0 \}.
\end{align*} with an almost global domain of attraction. \\

This completes the proof of asymptotic stability of estimation error state $(\Bar{g}_i, \varphi_i) =
(I, 0)$ with an almost global domain of attraction.
\end{proof}

\section{Simulation Results}\label{sec:6}

In order to numerically verify this estimator, simulated true states of an aerial vehicle are produced using a six degrees of freedom (6DOF) rigid body dynamics model. The continuous-time 6DOF rigid-body dynamics equations are as follows:
\begin{align}
     \dot{b} &= R\varv \nonumber \\
     m_v\dot{\varv} & = -m_v\Omega^{\times}\varv + \phi_v \nonumber \\
     \dot{R} & = R\Omega^{\times} \nonumber \\
     J_v\dot{\Omega} & = -\Omega^{\times}J_v\Omega + \tau_v,
 \end{align}
where $m_v$ and $J_v$ are the mass and moment of inertia of the rigid body, respectively. $\phi_v$ and $\tau_v$ are the total force and torque in the body-frame, respectively. A numerical method for the simulation of 6DOF rigid body dynamics is presented in \cite{baraff1997introduction}. Several efficient algorithms to compute rigid body dynamics can be found in \cite{featherstone2014rigid}. The vehicle mass and moment of inertia are taken to be $m_v = 0.42$ kg and $J_v = \text{diag}(10^{-3}\times[51.2 \; 60.2 \; 59.6])$ kg$\cdot$m$^2$, respectively. The resultant external forces and torques applied on the vehicle are $\phi_v(t) = 10^{-3}[10 \text{cos}(0.1t) \; 2\text{sin}(0.2t) \; -2\text{sin}(0.5t)]^\text{T}$ N and $\tau_v(t) = 10^{-6}\phi_v(t)$ N$\cdot$m, respectively. The flight area is assumed to be a cubic space of size $20m\times 20m\times 20 m$ with the origin of the inertial frame located at the center of this cube. The initial attitude and position of the vehicle are
$$ R_0 = \text{expm}_{\SO}\left(\left( \frac{\pi}{4}\times\left[\frac{3}{7} \; -\frac{6}{7} \; \frac{2}{7}\right]\T \right)^\times\right),$$ $$\text{and} \;\; b_0 = [2.5 \; 0.5 \; -3]\T \; m. $$
The vehicle’s initial angular and translational velocities are
$$ \Omega_0 = [0.2 \; -0.05 \; 0.1]\T rad/s,$$ $$\text{and} \;\; \varv_0 = [-0.05 \; 0.15 \; 0.03]\T \; m/s. $$

The vehicle dynamics is simulated over a time interval of T = 60 s, with a time step-size of h = 0.01 s. The trajectory
of the vehicle over this time interval is depicted in Fig-2. The following two inertial directions, corresponding to Nadir and Earth’s magnetic field direction are measured by the inertial sensors on the vehicle:
$$ d_1 = [0 \; 0 \; -1]\T, \qquad d_2 = [0.1 \; 0.975 \; -0.2]\T. $$
For optical measurements, 8 beacons are placed at the corners of the room. It has been assumed that at least two beacons are measured at each time instant. The observed directions in the body-fixed frame are simulated with the help of the aforementioned true states. The true quantities are disturbed by bounded, random noise with zero mean to simulate realistic measurements. Based on coarse attitude sensors like sun sensors, optical sensors and magnetometers, a random noise bounded in magnitude by $2.4^\circ$ is added to the matrix $L = R\T D$ to generate measured $L^m$. Similarly, random noise bounded in magnitude by $0.97^\circ/s$ and $0.025m/s$, which are close to actual noise levels of coarse rate gyros, are added to $\Omega$, $\varv$ to generate measured $\Omega^m$ and $\varv^m$ respectively. The scalar gain $m = 1.5$ and the dissipation term is chosen to be, $l = 0.1$. Furthermore, the value of the gains $k_p$ and $\kappa$ are chosen to be $k_p = 150$ and $\kappa = 100$. The state estimates are initialized at:
$$ \hat{g} = I, \quad \hat{\Omega}_0 = [0.1 \; 0.45 \; 0.05]\T rad/s,$$ $$\text{and} \;\; \hat{\varv}_0 = [2.05 \; 0.64 \; 1.29]\T \; m/s.$$\\
It is to be noted that the estimation scheme presented in Section \ref{sec:5} is a discrete-time almost global asymptotically stable estimation scheme for simultaneous estimation of position, attitude, linear velocity, and angular velocity. The filter equations in \eqref{eq:53} are implicit and therefore have to be solved simultaneously at each time instant in order to obtain estimates for that time instant. We start with a set of random initial state estimates as given above. With the help of the procedure mentioned in the previous paragraph, true values of inertial and optical directions, angular velocity, and linear velocity are disturbed with random noise to generate realistic measurements and these measurements are fed into the filter equations at each time instant in real-time. The equations in \eqref{eq:53} are simultaneously solved with the help of \verb++fsolve available in Matlab at each time instant to generate estimated values of state trajectories for 60s. The estimates are then compared with initially generated true state trajectories. The position and attitude estimation error are shown in fig. \ref{fig:3a} and fig. \ref{fig:3b} respectively. We observe that both the position and the attitude errors converge to a bounded set around the equilibrium after about 30s. The size of the bounded set is dictated by the noise magnitudes. The corresponding velocity error plots are shown in fig. \ref{fig:3c} and fig. \ref{fig:3d} respectively, and show desired performance. It is important to note that we are not required to assume the noise distribution properties. 

\begin{figure}
	\centering
	\includegraphics[scale = 0.65]{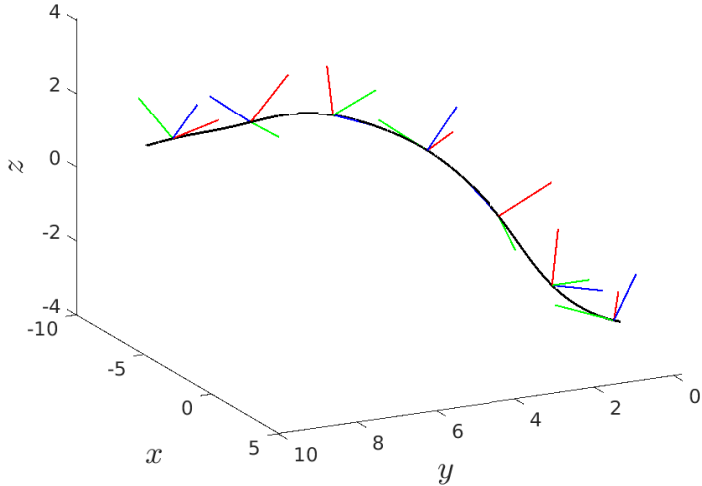}
	\vspace*{-4mm}
	\caption{Trajectory of the body}
	\label{trajectory}
\end{figure}

\begin{figure}
\begin{subfigure}{0.5\textwidth}
	\centering
	\includegraphics[width = 1\textwidth]{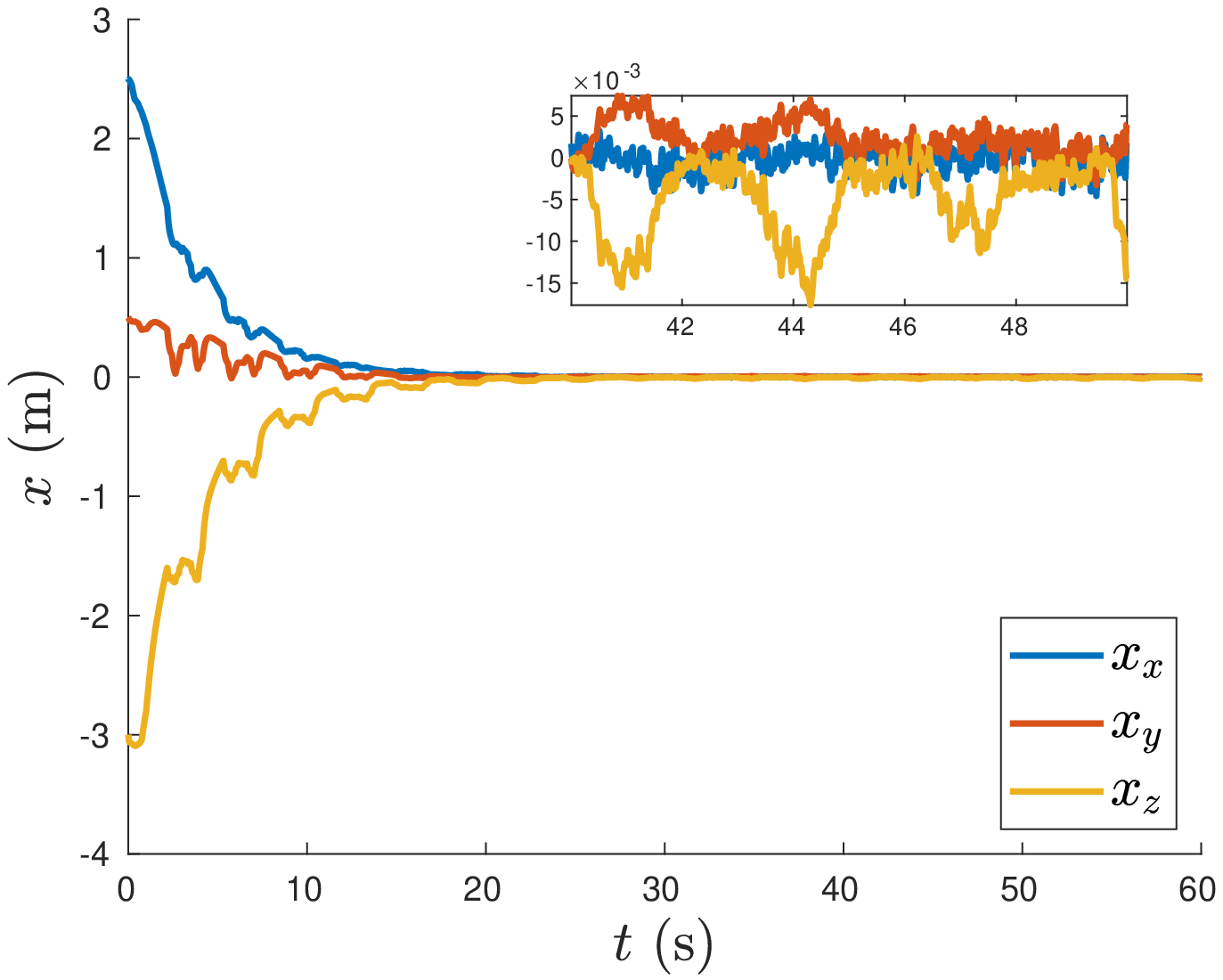}
	\caption{Position estimation error}
	\label{fig:3a}
\end{subfigure}
\begin{subfigure}{0.5\textwidth}
	\centering
	\includegraphics[width = 1\textwidth]{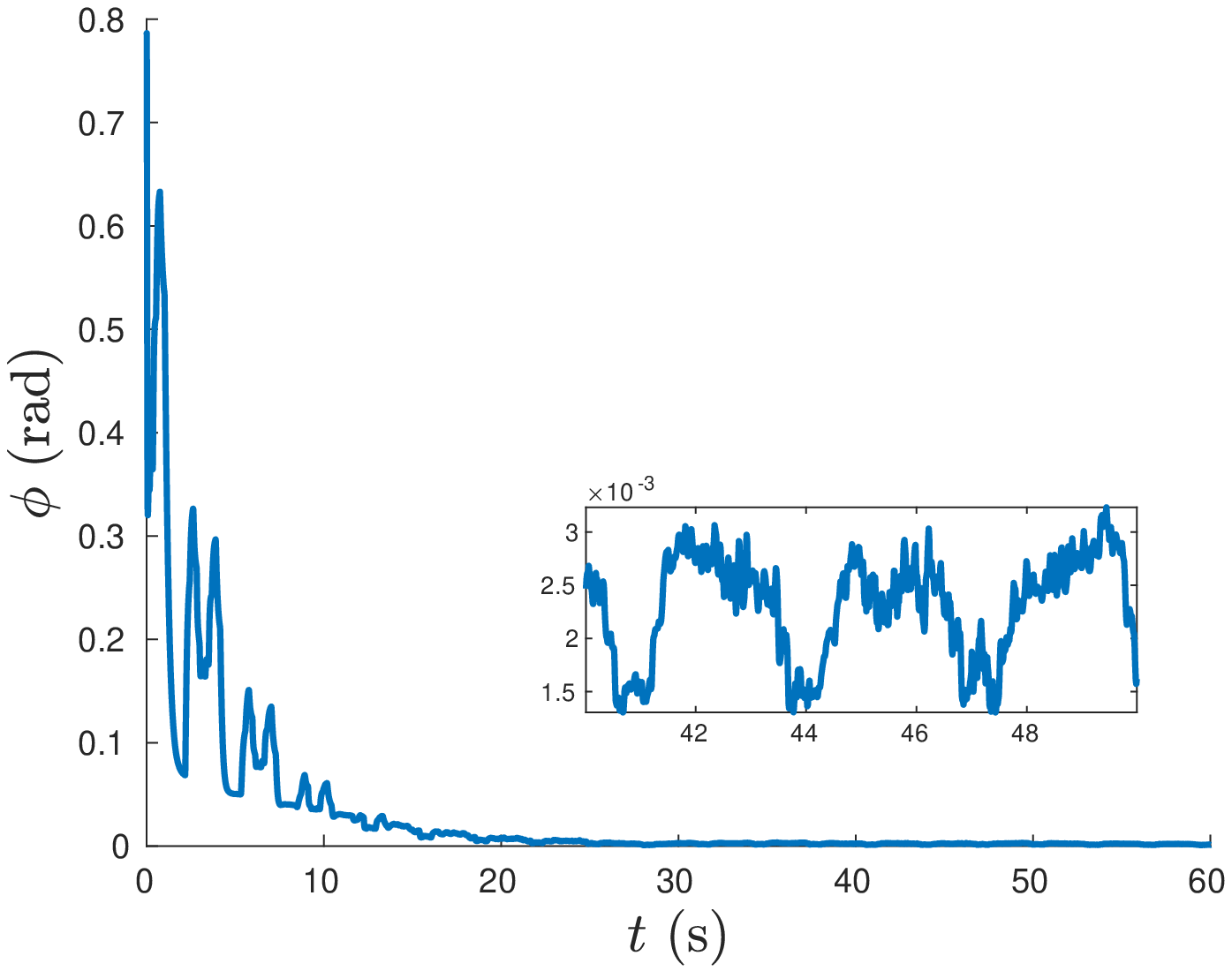}
	\caption{Principle angle of the attitude estimation error}
	\label{fig:3b}
\end{subfigure}
\begin{subfigure}{0.5\textwidth}
	\centering
	\includegraphics[width = 1\textwidth]{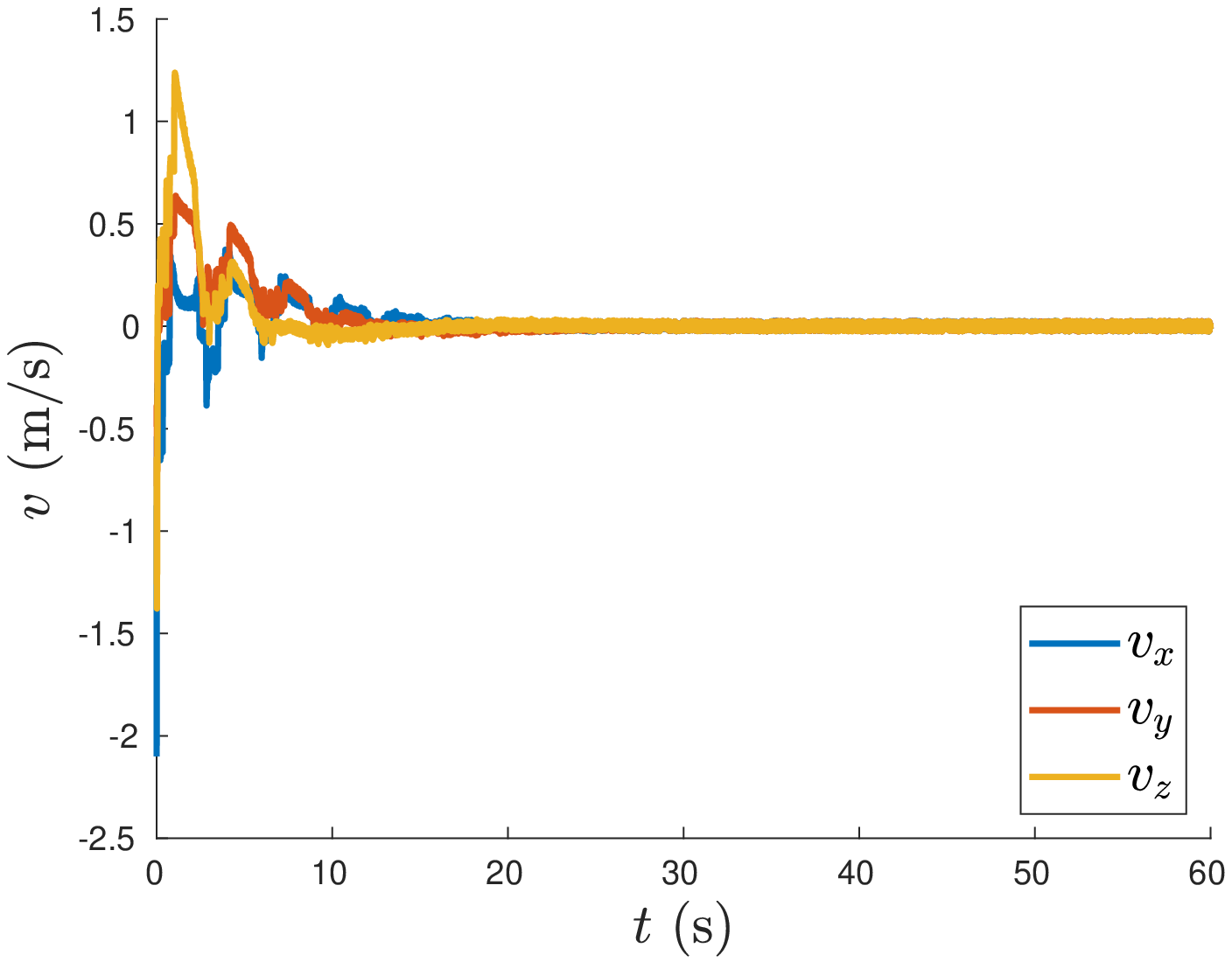}
	\caption{Translational velocity estimation error}
	\label{fig:3c}
\end{subfigure}
\begin{subfigure}{0.5\textwidth}
	\centering
	\includegraphics[width = 1\textwidth]{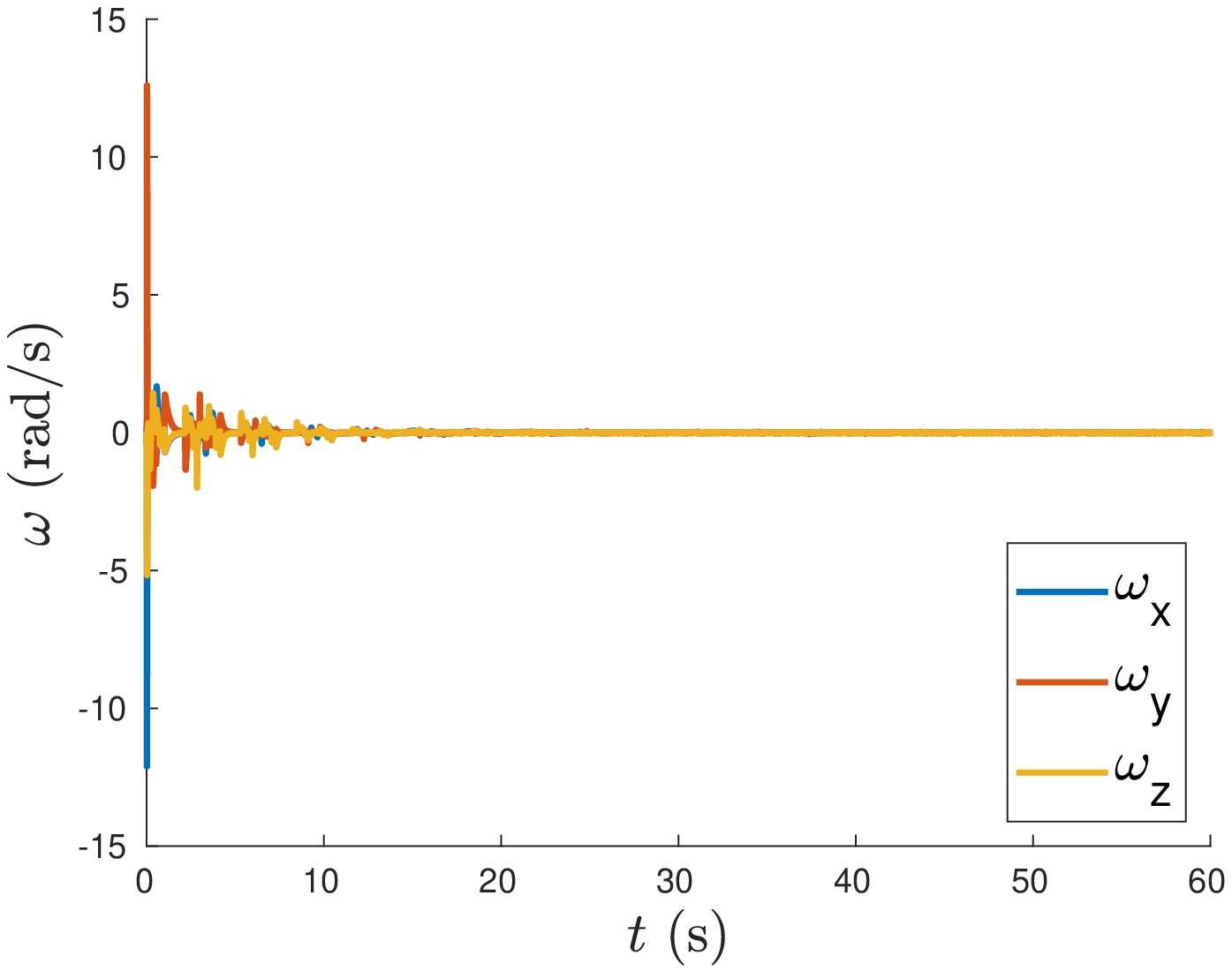}
	\caption{Angular velocity estimation error}
	\label{fig:3d}
\end{subfigure}
\caption{}
\label{fig3}
\end{figure}


\section{Conclusions}\label{sec:7}
An asymptotically stable and optimal discrete-time rigid body pose estimator has been presented in this work. Suitable artificial potential energy and kinetic energy-like functions of state estimation errors were used to construct a Lagrangian in discrete time. The discrete Lagrange-d'Alembert principle was applied to this Lagrangian to obtain an optimal filtering scheme. The dissipation terms were calculated through a discrete-Lyapunov analysis carried out on a Morse-Lyapunov function that corresponds to the total energy function constructed from the kinetic and potential energy-like terms used to construct the Lagrangian. The theoretical assertions are supported through realistic numerical simulations. It has been observed that the estimation errors converge to a bounded neighborhood of the true states. The rates of convergence and domain of convergence can be controlled by changing scalar gains associated with the potential and kinetic energy-like terms that make up the Lagrangian.
Future work in this domain would look into designing an explicit filtering scheme by constructing a suitable cost function, so that numerical computations are faster and simpler for onboard implementation.


%
%

\bibliographystyle{spmpsci}      
\bibliography{bibliography}   

\begin{thebibliography}{10}
\providecommand{\url}[1]{{#1}}
\providecommand{\urlprefix}{URL }
\expandafter\ifx\csname urlstyle\endcsname\relax
  \providecommand{\doi}[1]{DOI~\discretionary{}{}{}#1}\else
  \providecommand{\doi}{DOI~\discretionary{}{}{}\begingroup
  \urlstyle{rm}\Url}\fi

\bibitem{amelin2014algorithm}
Amelin, K., Miller, A.: An algorithm for refinement of the position of a light
  uav on the basis of kalman filtering of bearing measurements.
\newblock Journal of Communications Technology and Electronics \textbf{59}(6),
  622--631 (2014)

\bibitem{baraff1997introduction}
Baraff, D.: An introduction to physically based modeling: rigid body simulation
  i—unconstrained rigid body dynamics.
\newblock SIGGRAPH course notes \textbf{82} (1997)

\bibitem{bhatt2020optimal}
Bhatt, M., Sanyal, A.K., Sukumar, S.: Optimal multi-rate rigid body attitude
  estimation based on lagrange-d'alembert principle.
\newblock arXiv preprint arXiv:2008.04104  (2020)

\bibitem{bhatt2020rigid}
Bhatt, M., Sukumar, S., Sanyal, A.K.: Rigid body geometric attitude estimator
  using multi-rate sensors.
\newblock In: 2020 59th IEEE Conference on Decision and Control (CDC), pp.
  1511--1516. IEEE (2020)

\bibitem{featherstone2014rigid}
Featherstone, R.: Rigid body dynamics algorithms.
\newblock Springer (2014)

\bibitem{filipe2015extended}
Filipe, N., Kontitsis, M., Tsiotras, P.: Extended kalman filter for spacecraft
  pose estimation using dual quaternions.
\newblock Journal of Guidance, Control, and Dynamics \textbf{38}(9), 1625--1641
  (2015)

\bibitem{izadi2014rigid}
Izadi, M., Sanyal, A.K.: Rigid body attitude estimation based on the
  lagrange--d’alembert principle.
\newblock Automatica \textbf{50}(10), 2570--2577 (2014)

\bibitem{izadi2016rigid}
Izadi, M., Sanyal, A.K.: Rigid body pose estimation based on the
  lagrange--d’alembert principle.
\newblock Automatica \textbf{71}, 78--88 (2016)

\bibitem{lasalle1976stability}
LaSalle, J.: The Stability of Dynamical Systems, vol.~25.
\newblock SIAM (1976)

\bibitem{mahony2008nonlinear}
Mahony, R., Hamel, T., Pflimlin, J.M.: Nonlinear complementary filters on the
  special orthogonal group.
\newblock IEEE Transactions on automatic control \textbf{53}(5), 1203--1218
  (2008)

\bibitem{marsden2001discrete}
Marsden, J.E., West, M.: Discrete mechanics and variational integrators.
\newblock Acta Numerica \textbf{10}, 357--514 (2001)

\bibitem{rehbinder2003pose}
Rehbinder, H., Ghosh, B.K.: Pose estimation using line-based dynamic vision and
  inertial sensors.
\newblock IEEE Transactions on Automatic control \textbf{48}(2), 186--199
  (2003)

\bibitem{vasconcelos2007landmark}
Vasconcelos, J.F., Cunha, R., Silvestre, C., Oliveira, P.: Landmark based
  nonlinear observer for rigid body attitude and position estimation.
\newblock In: 2007 46th IEEE Conference on Decision and Control, pp.
  1033--1038. IEEE (2007)

\bibitem{vasconcelos2010nonlinear}
Vasconcelos, J.F., Cunha, R., Silvestre, C., Oliveira, P.: A nonlinear position
  and attitude observer on se (3) using landmark measurements.
\newblock Systems \& Control Letters \textbf{59}(3-4), 155--166 (2010)

\bibitem{vasconcelos2008nonlinear}
Vasconcelos, J.F., Silvestre, C., Oliveira, P.: A nonlinear gps/imu based
  observer for rigid body attitude and position estimation.
\newblock In: 2008 47th IEEE Conference on Decision and Control, pp.
  1255--1260. IEEE (2008)

\bibitem{vertechy2007accurate}
Vertechy, R., Castelli, V.P.: Accurate and fast body pose estimation by three
  point position data.
\newblock Mechanism and machine theory \textbf{42}(9), 1170--1183 (2007)

\bibitem{wahba1965least}
Wahba, G.: A least squares estimate of satellite attitude.
\newblock SIAM review \textbf{7}(3), 409--409 (1965)

\bibitem{zamani2013minimum}
Zamani, M., Trumpf, J., Mahony, R.: Minimum-energy filtering for attitude
  estimation.
\newblock IEEE Transactions on Automatic Control \textbf{58}(11), 2917--2921
  (2013)

\end{thebibliography}


\end{document}